\documentclass[12pt,reqno]{amsart}
\usepackage{jrmacros}

\usepackage{mathtools}
\mathtoolsset{showonlyrefs}

\usepackage{appendix}

\theoremstyle{plain}
\newtheorem{theorem}{Theorem}[section]
\newtheorem{proposition}[theorem]{Proposition}
\newtheorem{corollary}[theorem]{Corollary}
\newtheorem{lemma}[theorem]{Lemma}
\newtheorem{conjecture}[theorem]{Conjecture}

\newtheorem{congruence}[theorem]{Congruence}

\newtheorem*{theorem*}{Theorem}
\newtheorem*{proposition*}{Proposition}
\newtheorem*{corollary*}{Corollary}
\newtheorem*{lemma*}{Lemma}
\newtheorem*{conjecture*}{Conjecture}

\theoremstyle{definition}
\newtheorem{definition}[theorem]{Definition}

\newtheorem*{definition*}{Definition}
\newtheorem*{example*}{Example}
\newtheorem*{question*}{Question}
\newtheorem*{philosophy*}{Philosophy}

\theoremstyle{remark}
\newtheorem{remark}[theorem]{Remark}

\newtheorem*{remark*}{Remark}

\def\ape{Ap\'{e}ry}

\def\hp{h_{p}}

\def\p{\mathbf{p}}

\def\Ai{\Q_{p\to\infty}}

\def\ma{\cM}

\def\mf{\fil}

\def\fil{\mathrm{Fil}}
\def\p{\mathbf{p}}

\def\bu{\underline{\mathbf{u}}}

\def\mhs{\cM}
\def\mzv{multiple zeta value}

\def\ppc{supercongruence}

\begin{document}
\title{The MHS algebra and supercongruences}
\author{Julian Rosen}
\email{julianrosen@gmail.com}
\date{\today}
\maketitle

%\begin{abstract}
%
%\end{abstract}

\begin{abstract}
A supercongruence is a congruence between rational numbers modulo a power of a prime. In this paper, we give a technique for finding and algorithmically proving supercongruences by expressing terms as infinite series involving certain generalizations of the harmonic numbers. We apply the technique to derive many new supercongruences. We also provide software for finding and proving supercongruences using our technique.
\end{abstract}

\setcounter{tocdepth}{1}
\tableofcontents

%\newpage
\section{Introduction}
\subsection{Supercongruences}

If $x$ and $y$ are rational numbers and $k$ is a positive integer, we write
$
x\equiv y\mod k
$
if the numerator of $x-y$ is divisible by $k$. A \emph{\ppc{} family} (or just  \emph{supercongruence}) is a statement of the form
\begin{equation}
\label{defppc}
a_p\equiv b_p\mod p^n\gap\text{for all but finitely many primes $p$},
\end{equation}
where $a_p$ and $b_p$ are rational (or $p$-adic) numbers depending on $p$ and $n$ is a positive\footnote{Traditionally \eqref{defppc} is called a supercongruence only when $n$ is at least $2$. We also allow $n=1$ for maximal generality} integer.

%\subsubsection{Examples}
%One well-known example is a congruence for binomial coefficients, which says that for all integers $m\geq n\geq 0$, the congruence
%\begin{equation}
%\label{bincong}
%{mp\choose np}\equiv {m\choose n}\mod p^3
%\end{equation}
%holds for all primes $p\geq 5$.
%
%Another example comes from \ape's 1979 proof of the irrationality of $\zeta(3)$. The proof involves a sequence of numbers $a_n$, now called the \emph{\ape{} numbers}, defined by the recurrence relation
%\[
%a_n=\frac{34n^3-51n^2+27n-5}{n^3} \,a_{n-1}-\frac{(n-1)^3}{n^3} \,a_{n-2},
%\]
%along with the initial conditions  $a_0=1$, $a_1=5$. The terms $a_n$ turn out to be integers (this is not obvious from the definition), and they have interesting arithmetic properties. For example, it is known that the congruence
%\begin{equation}
%\label{introape}
%a_{p-1}\equiv 1\mod p^3
%\end{equation}
%holds for all primes $p\geq 5$.

%\Ppc{}s often encode subtle arithmetic information. For example, it is known that \eqref{bincong} holds modulo $p^4$ if and only if $p$ divides the numerator of the Bernoulli number $B_{p-3}$ (the same is true for \eqref{introape} by Congruence \ref{cz1} and \eqref{eqzetak} below). Theorems of Herbrand \cite{Her32} and Ribet \cite{Rib76} imply this condition is equivalent to the condition that there exists a non-trivial $p$-torsion element in the class group of the $p$-th cyclotomic field on which $\gal(\Q(\zeta_p)/\Q)=\F_p^\times$ acts by the negative third power.

\subsubsection{Multiple harmonic sums}
\label{intromhs}
A \emph{composition} is a finite ordered list of positive integers. For $N$ a positive integer and $\bs=(s_1,\ldots,s_k)$ composition, the quantity
\[
H_N(s_1,\ldots,s_k):=\sum_{N\geq n_1>\ldots>n_k\geq 1}\frac{1}{n_1^{s_1}\ldots n_k^{s_k}}\in\Q,
\]
is called a \emph{multiple harmonic sum} (\emph{MHS}). We call $|\bs|:=s_1+\ldots,s_k$ the \emph{weight} and $k$ the \emph{depth}. Multiple harmonic sums are known to satisfy many \ppc{}s. The case $N=p-1$, with $p$ a prime, is particularly rich. For instance, if $n$ is a positive integer, then $\H(n)\equiv 0\mod p$, and if $n$ is odd then in fact $\H(n)\equiv 0\mod p^2$. Another example is the reversal congruence, which appeared in \cite{Hof04a}:
\[
\H(s_1,\ldots,s_k)\equiv (-1)^{s_1+\ldots+s_k}\H(s_k,\ldots,s_1)\mod p.
\]

Multiple harmonic sums are sometimes viewed as basic units into which more complicated congruences can be decomposed. As an example, consider the $p$-th central binomial coefficient ${2p\choose p}$, for which several \ppc{}s are known. We have
\begin{align}
{2p\choose p}&=2\frac{(p+1)\cdot (p+2)\cdots(p+(p-1))}{1\cdot2\cdots (p-1)}\\
&=2\lp 1+\frac{p}{1}\rp\lp1+\frac{p}{2}\rp\cdots\lp 1+\frac{p}{p-1}\rp\\
&=2\sum_{n=0}^{p-1}p^n\H(1^n).\label{binex}\\
\end{align}
Here we have used the shorthand notation
\[
H_N(1^n):=H_N(\underbrace{1,\ldots,1}_n).
\]
So, for instance, the well-known Wolstenholme congruence
\[
{2p\choose p}\equiv 2\mod p^3
\]
is equivalent to the multiple harmonic sum supercongruence \[
p\H(1)+p^2\H(1,1)\equiv 0\mod p^3.
\]
In Sections \ref{secred}, \ref{secmix}, and \ref{sectec}, we compute expansions similar to \eqref{binex}, expressing many other quantities in terms of multiple harmonic sums.

\begin{remark}
The limits of multiple harmonic sums $H_N(\bs)$ as $N\to\infty$, when they exist, are called \emph{multiple zeta values}. The multiple zeta values are examples of periods (definite integrals of rational functions with rational coefficients over rationally-defined regions), and the theory of motives predicts periods should satisfy a version of Galois theory (see \cite{And09}). In \cite{Ros16b}, the author uses the Galois theory of multiple zeta values to construct a Galois theory of \ppc{}s.
\end{remark}

%We are especially interested in the following class of \ppc{} for multiple harmonic sums.
%\begin{definition}
%A \emph{mixed MHS congruence} (or just \emph{mixed congruence}) is a statement of the form
%\begin{equation}
%\label{eqinh}
%\sum_{i=1}^m a_i p^{b_i}\H(\bs_i)\equiv 0\mod p^n\gap\text{for all but finitely many $p$},
%\end{equation}
%where $a_1,\ldots, a_m\in\Q$, $b_1,\ldots,b_m\in\Z$, and $\bs_1,\ldots,\bs_m$ are compositions, all of which are independent of $p$.
%\end{definition}

\subsection{Results}
The main result of this paper is a technique for discovering and algorithmically proving \ppc{}s by expanding terms into $p$-adically convergent series involving the multiple harmonic sums $\H(\bs)$. We apply our technique to derive many new results, including congruences for binomial coefficients (Congruence \ref{cb}), the \ape{} numbers (Congruences \ref{ca1}, \ref{cz1}, \ref{cz2}), Bernoulli numbers and $p$-adic zeta values (Congruences \ref{cz1} and \ref{cz2}), alternating harmonic numbers (Congruence \ref{congalt}), some sums involving harmonic numbers (Congruences \ref{cs1}, \ref{cs2}), $p$-restricted harmonic numbers (Congruences \ref{cr1}, \ref{cr3}), and certain `curious' generalized harmonic numbers (Congruence \ref{cc}). A summary of the technique is given in Section \ref{ssmix}.

We construct a commutative $\Q$-algebra $\mhs$, called the MHS algebra, consisting of prime-indexed sequences admitting an expansion of a certain form in terms of multiple harmonic sums. We give an algorithm for proving \ppc{}s between elements of the MHS algebra.
%Because we are interested in congruences modulo powers of $p$, it is natural to view the quantities appearing as elements of $\Q_p$. If two prime-indexed families $a_p,b_p\in\Q_p$ satisfy the condition that for every $n$, the congruence $a_p\equiv b_p\mod p^n$ holds for all sufficiently large $p$, then from the standpoint of \ppc{}s $a_p$ and $b_p$ are indistinguishable. Thus it is convenient to consider the tuples $(a_p)$ as elements of the ring
%\[
%\Ai:=\frac{\left\{(a_p)\in\prod_p\Q_p:v_p(a_p)\text{ bounded below}\right\}}{\left\{(a_p)\in\prod_p\Q_p:v_p(a_p)\to\infty \text{ as }p\to\infty\right\}}.
%\]
%
%We define a subring $\mhs\subset\Ai$ called the MHS algebra, consisting of those quantities admitting an expansion in terms of multiple harmonic sums appropriate for our technique.
We also provide software for finding and proving \ppc{} using this algorithm.
% between elements of the MHS algebra.
A description of the software is given in Appendix \ref{appsoft}. The algorithm produces unconditional proofs of supercongruences, and the truth of Conjecture \ref{conjds} would imply the algorithm is in fact a decision procedure for \ppc{}s between elements of the MHS algebra.

\subsection{Outline}
In Sections \ref{secred} and \ref{secmix}, we prove many new \ppc{}s using expansions in terms of multiple harmonic sums. In Section \ref{ssmix} we summarize our technique for proving \ppc{}s using these expansion. In Section \ref{secalg} we define the MHS algebra, which is the class of quantities for which our technique is algorithmically applicable.

In Section \ref{sectec} we compute a number of additional expansions in terms of multiple harmonic sums, which allow us to algorithmically prove many additional \ppc{}s.

In Section \ref{secmpl} we discuss a generalization to truncated multiple polylogarithms.

We provide software implementing the technique of this paper, which we describe in Appendix \ref{appsoft}. Appendix \ref{appdec} contains a description of a family of \ppc{}s for multiple harmonic sums, which appared recently in \cite{Jar16}.

\subsection{Recent related work}
Recent work of Jarossay \cite{Jar16} gives several series identities for multiple harmonic sums $\H(\bs)$, which are used in this paper. Additionally, \cite{Jar16,Jar16a,Jar16b} establish a relationship between multiple harmonic sums and $p$-adic \mzv{}s. In \cite{Ros16b} the author uses this relationship to construct a Galois theory of \ppc{}s for the MHS algebra.

Several recent works involve computer algorithms for proving congruences. Rowland and Yassawi \cite{Row13} given an automatic method for proving congruences for diagonal coefficients of multi-variate rational power series. The results of \cite{Row13} are generalized by Rowland and Zeilberger in \cite{Row14}. A very recent paper of Chen, Hou, and Zeilberger \cite{Che16} gives an algorithm for proving congruences modulo $p$ for certain power series coefficients.

The present work involves infinite, $p$-adically convergent series identities involving multiple harmonic sums, and in Section \ref{secmpl} we look at related series involving truncated multiple polylogarithms. A recent work of Seki \cite{Sek16} investigates a closely related $p$-adic series identity for truncated multiple polylogarithms.

%%%%%%%
%
%   Multiple harmonic sums
%
%%%%%%%%
\section{Congruences for multiple harmonic sums}
\label{secmhs}
In this section, we describe a class of \ppc{}s for multiple harmonic sums. There is an algorithm for proving \ppc{}s in this class, which is conjecturally a decision procedure.
%, which we state informally as Conjecture \ref{coninf} and formally as Conjecture \ref{confor}.

\subsection{Stuffle product}
It is well-known that the product of two multiple harmonic sums with the same bound of summation can be written as a sum of multiple harmonic sums with the same bound. This is best illustrated with an example:
\begin{align}
H_N(2)H_N(3)&=\sum_{n=1}^N\sum_{m=1}^N\frac{1}{n^2m^3}\\
&=\lp\sum_{n>m}+\sum_{n<m}+\sum_{n=m}\,\rp\frac{1}{n^2m^3}\\
&=H_N(2,3)+H_N(3,2)+H_N(5).\label{stuff}
\end{align}
The expression on the right is called a \emph{stuffle product} of $H_N(2)$ and $H_N(3)$, which is a combination of shuffling and stuffing (the $2$ and $3$ are \emph{shuffled} to give $(2,3)$ and $(3,2)$, and \emph{stuffed} to give $(5)$).

\subsection{Weighted congruences}
In Section \ref{intromhs}, we saw the congruences
\begin{gather}
\H(n)\equiv 0\mod p,\text{ $n$ even},\\
\H(n)\equiv 0\mod p^2,\text{ $n$ odd},\\
\H(s_1,\ldots,s_k)\equiv (-1)^{s_1+\ldots+s_k}\H(s_k,\ldots,s_1)\mod p.
\end{gather}
These congruences are \emph{homogeneous}, in that each involves only compositions of a single weight.
There are also inhomogeneous congruences, for instance
\[
\H(1)\equiv \frac{1}{3}p^2\H(2,1)\mod p^4,
\]
or the more complex
\[
33 \H(2)\equiv 22 p\H(2,1)+2p^2\H(2,1,1)\mod p^4.
\]
It these examples, terms $\H(\bs)$ are multiplied by an explicit power of $p$, which increase with the weight of $\bs$. This is typical for \ppc{}s involving $\H$. It is convenient to multiply through by a power of $p$ so that each term has the form $p^{|\bs|}\H(\bs)$.

\begin{definition}(\cite{Ros13})
For each composition $\bs$, we define the \emph{weighted multiple harmonic sum}
\[
\hp(\bs):=p^{|\bs|}\H(\bs).
\]
A \emph{weighted MHS congruence} (or just \emph{weighted congruence}) is a statement of the form
\begin{equation}
\label{defwc}
\sum_{i=1}^m \alpha_i \hp(\bs_i)\equiv 0\mod p^n\gap\text{ for all but finitely many }p,
\end{equation}
 where $n\in\Z_{>0}$, the coefficients $\alpha_1,\ldots,\alpha_m\in\Q$, and the compositions $\bs_1,\ldots,\bs_m$ are independent of $p$.
\end{definition}

Note that $\hp(\bs)\equiv 0\mod p^{|\bs|}$.

\subsection{Extension} Frequently weighted congruences can be strengthened. For example the now-well-known congruence $\hp(1,1)\equiv 0\mod p^3$ can be strengthened to $3\hp(1,1)+\hp(2,1)\equiv 0\mod p^4$, and further to $3\hp(1,1)+\hp(2,1)+\hp(3,1)\equiv 0\mod p^5$. These congruences arise from a convergent $p$-adic series identity
\begin{equation}
\label{inf}
3\hp(1,1)+\sum_{n\geq 2}\hp(n,1)=0.
\end{equation}
Series identities like \eqref{inf} are ubiquitous. Another example, which is not difficult to check, is
\begin{equation}
\label{inf2}
2\hp(1)+\sum_{n\geq 2}\hp(n)=0,
\end{equation}
from which follows a family of weighted congruences:
\begin{align*}
2\hp(1)&\equiv 0\mod p^2\\
2\hp(1)+\hp(2)&\equiv 0\mod p^3\\
2\hp(1)+\hp(2)+\hp(3)&\equiv 0\mod p^4\\
&\vdots
\end{align*}
In \cite{Ros13}, the author gave two families of $p$-adic series identities like \eqref{inf}, \eqref{inf2}. Jarossay \cite{Jar16} recently gave more general methods for generating these identities. We expect that every weighted congruence is a consequence of one of these series identities (this Conjecture A of \cite{Ros13}), and that moreover every such series identities is a consequence of the identities in \cite{Jar16} (this is Conjecture \ref{conjds} of the present work). One of the identity families of \cite{Jar16} is described in Appendix \ref{appdec}.

\begin{remark}
\label{fact}
The identities described in Appendix \ref{appdec} can be computed algorithmically, so there is an algorithm for proving weighted congruences. We expect (Conjecture \ref{conjds}) that \emph{every} weighted congruence can be proven using this algorithm.
\end{remark}

\section{Reduction to multiple harmonic sums}
\label{secred}

It is sometimes possible to prove a \ppc{} by reducing it to a weighted MHS congruence. We illustrate this technique with three examples.

\subsection{Binomial coefficients}
We prove the following congruence.

 \begin{congruence}
 \label{cb}
 For $p$ odd:
\begin{equation}
\label{exa}
12-9{2p\choose p}+2{3p\choose p}\equiv 24p^3\sum_{n=1}^{p-1}\frac{1}{n^3}\mod p^6.
\end{equation}
\end{congruence}
That the left hand side is divisible by $p^3$ follows from the well-known congruence for positive integers $k$, $r$,
\[
{kp\choose rp}\equiv {k\choose p}\mod p^3.
\]
\begin{proof}
Equation \eqref{binex} above is the identity ${2p\choose p}=2\sum_{n=0}^{p-1}\hp(1^n)$. It follows that ${2p\choose p}\equiv 2\sum_{n=0}^5\hp(1^n)\mod p^6$. A similar computation gives ${3p\choose p}\equiv 3\sum_{n=0}^5 2^n \hp(1^n)\mod p^6$. So \eqref{exa} is equivalent to the weighted congruence
\begin{equation}
\label{wc}
\sum_{n=1}^5\lp6\cdot 2^n-18\rp \hp(1^n)\equiv 24\,\hp(3)\mod p^6.
\end{equation}
There are two approaches to derive \eqref{wc}.
\begin{itemize} \item The first approach is to piece together known results from the literature. We start with $\hp(1^5)\equiv 0$ and $\hp(1^3)\equiv 2\hp(1^4)\mod p^6$ (\cite{Zha08}, Theorem 1.6 ). The identity
\[
\lp 1-\frac{p}{1}\rp\cdots\lp 1-\frac{p}{p-1}\rp=1
\]
implies that
\[
\hp(1)-\hp(1^2)+\hp(1^3)-\hp(1^4)\equiv 0\mod p^6
\]
(this appeared in \cite{Ros12a}, Proposition 2.1). Together, these congruences show that \eqref{exa} is equivalent to
\[
72\hp(1^3)\equiv 24\hp(3),
\]
which also follows from \cite{Zha08}, Theorem 1.6.

\item The second approach is to verify \eqref{wc} by computer, using Remark \ref{fact}. 
\end{itemize}
\end{proof}
There is nothing special about the modulus $p^6$ in \eqref{exa}: the expression for the binomial coefficients in terms of the multiple harmonic sums $\hp(1^n)$ can be extended to hold modulo any desired power of $p$. In particular, it can be shown using the same technique that the difference between the left hand side and the right hand side of \eqref{exa} is congruent to $-48\hp(4,1,1)$ modulo $p^7$. Similar computations are possible for binomial coefficients of various other shapes, see Section \ref{ssbin}.

%%%%%%%%%%%
%
%   Apery numbers
%
%%%%%%%%%%%
\subsection{\ape{} numbers}
\label{ssape}
In 1979 \ape{} proved that $\zeta(3)$ is irrational (see \cite{Ape81}). The proof involves a sequence of positive integers $b_n$, now called \emph{\ape{} numbers}, defined by the recurrence relation $b_0=1$, $b_1=5$, and
\[
b_n=\frac{34n^3-51n^2+27n-5}{n^3} \,b_{n-1}-\frac{(n-1)^3}{n^3} \,b_{n-2}
\]
for $n\geq 2$. From the recurrence relation it is not obvious that the \ape{} numbers are integers; this follows from an expression for $b_n$ in terms of binomial coefficients:
\[
b_n=\sum_{k=0}^n {n\choose k}^2{n+k\choose k}^2.
\]
In the introduction, we stated the congruence
\begin{equation}
\label{eqacong}
b_{p-1}\equiv 1\mod p^3,
\end{equation}
which is given in \cite{Cho80}. We prove an extension modulo $p^5$.
\begin{congruence}
\label{ca1}
\label{propape}
For every prime $p\geq 7$, the \ape{} numbers satisfy
\begin{equation}
\label{apecong}
{2p\choose p}\cdot b_{p-1}\equiv 2\mod p^5.
\end{equation}
\end{congruence}

Note that ${2p\choose p}\equiv 2\mod p^3$, so \eqref{apecong} is a generalization of \eqref{eqacong}. The key tool is the following lemma, which allow us to prove a wide range of congruences for $b_{p-1}$.

\begin{lemma}
\label{ape}
There exists an algorithmically computable sequences of integer coefficients $\alpha_{\bs}$, indexed by the compositions, such that for every prime $p$, we have
\begin{equation}
\label{ra}
b_{p-1}=\sum_{\bs}\alpha_{\bs}\hp(\bs).
\end{equation}
The coefficients $\alpha_{\bs}$ have the property that for every $p$, there are only finitely many compositions $\bs$ of length less than $p$ with $\alpha_{\bs}\neq0$, so the right hand side of \eqref{ra} is actually a finite sum.
\end{lemma}

\begin{proof}
For $n$ an arbitrary positive integer, we compute
\begin{align*}
b_{n-1}&=\sum_{k=0}^{n-1} {n-1\choose k}^2{n+k-1\choose k}^2\\
&=1+\sum_{k=1}^{n-1}\lp\frac{n-k}{k}\rp^2\left[ \lp1-\frac{n}{1}\rp\ldots\lp1-\frac{n}{k-1}\rp\right]^2\cdot\\&\hspace{60mm}\frac{n^2}{k^2}\left[ \lp1+\frac{n}{1}\rp\ldots\lp1+\frac{n}{k-1}\rp\right]^2\\
&=1+\sum_{k=1}^{n-1}\lp\frac{n-k}{k}\rp^2\frac{n^2}{k^2}\left[ \lp1-\frac{n^2}{1^2}\rp\ldots\lp1-\frac{n^2}{(k-1)^2}\rp\right]^2\\
&=1+\sum_{k=1}^{n-1}\lp\frac{n^4}{k^4}-2\frac{n^3}{k^3}+\frac{n^2}{k^2}\rp\lp\sum_{i\geq 0} (-1)^i n^{2i}H_{k-1} (2^i)\rp^2.
\end{align*}
The product of the final two terms inside the outer sum can be expanded as a linear combination of terms $n^{|\bs|}H_{k-1}(\bs)$ using the stuffle product, and multiplying by $n^a/k^a$ and summing over $k$ gives terms $n^{|\bs|+a}H_{n-1}(a,\bs)$.

Take $n=p$ to obtain the desired result.

\end{proof}

\begin{proof}[Proof of Congruence \ref{propape}]
Lemma \ref{ape} and the expression \eqref{binex} for ${2p\choose p}$ shows that
 \eqref{apecong} is equivalent to a weighted congruence. This weighted congruence is not too hard to prove by hand, but by Remark \ref{fact} it can also be verified by computer.
\end{proof}

%%%%%%%%%%%
%
%   Bernoulli numbers and the $p$-adic zeta function
%
%%%%%%%%%%%
\subsection{Bernoulli numbers and the $p$-adic zeta function}
\label{sszeta}
The Bernoulli numbers $B_n$, $n\geq 0$ are a sequence of rational numbers defined by the exponential generating function
\[
\frac{x}{e^x-1}=\sum_{n=0}^\infty \frac{B_n}{n!}x^n.
\]
They are known to satisfy many congruences. We mention here the \emph{Kummer congruence}, which says that for any prime $p$ and integer $k\geq 1$,
\[
(1-p^{n-1})\frac{B_n}{n}\equiv(1-p^{m-1})\frac{B_m}{m}\mod p^k
\]
holds whenever $m$ and $n$ are positive integers, neither divisible by $p-1$, satisfying $n\equiv m\mod (p-1)p^{k-1}$.

The Riemann zeta function takes rational values at the negative integers, and there is an explicit formula in terms of Bernoulli numbers:
\[
\zeta(-n)=-\frac{B_{n+1}}{n+1}.
\]
For $p$ a fixed prime, the Kummer congruences imply that the restriction of $(1-p^{-s})\zeta(s)$ (the zeta function with the Euler factor at $p$ removed) to the negative integers in a given residue class modulo $p-1$ is uniformly $p$-adically continuous. Since the negative integers in any residue class modulo $p-1$ are $p$-adically dense in the integers, we can extend this function to the positive integers.
\begin{definition}
For each integer $k\geq 2$, we define the \emph{$p$-adic zeta value}
\[
\zeta_p(k):=\hspace{-4mm}\lim_{\substack{n\to-\infty,\\n\equiv k\mod p-1,\\n\to k\text{ $p$-adically}}}(1-p^{-n})\zeta(n)\in\Q_p.
\]
\end{definition}
For all integers $k$, $n\geq 2$, and all primes $p\geq k+2$:
\begin{align}
\label{eqzetak}
\zeta_p(k)&\equiv \frac{B_{p-k}}{k}\mod p,\\
\zeta_p(k)&\equiv \frac{B_{p^{n-1}(p-1)+1-k}}{k-1}\lp1-\frac{p^{n-1}}{k-1}\rp\mod p^{n}.
\end{align}
The $p$-adic zeta values can be expressed in terms of the more general $p$-adic $L$-functions of Kubota-Leopoldt; we have
\[
\zeta_p(k)=L_p(k,\omega_p^{1-k}),
\]
where $\omega_p$ is the Teichm\"uller character.

The $p$-adic zeta values can be expressed in terms of weighted multiple harmonic sums.
\begin{theorem}[\cite{Was82}, Theorem 5.11]
For every integer $k\geq 2$, there is a $p$-adically convergent series identity
\begin{equation}
\label{z}
p^k \zeta_p(k)=\sum_{n\geq k-1}\frac{(-1)^{k+n+1}}{k-1}{n-1\choose k-2}B_{n+1-k}\hp(n).
\end{equation}
\end{theorem}
The most important feature of \eqref{z} is that the coefficients
\[
\frac{(-1)^{k+n+1}}{k-1}{n-1\choose k-2}B_{n+1-k}
\]
are independent of $p$. Combining \eqref{ra} with \eqref{z} and using Remark \ref{fact}, it is a finite computation to verify the following result relating the \ape{} numbers to $\zeta_p$.
\begin{congruence}
\label{cz1}
\begin{equation}
\label{apez}
b_{p-1}\equiv 1+2p^3\zeta_p(3)-16p^5\zeta_p(5)-14p^6\zeta_p(3)^2-100p^7\zeta_p(7)\mod p^8.
\end{equation}
\end{congruence}
\noindent We can similarly get expressions for $a_{p-1}$ in terms of Bernoulli numbers. We use modulus $p^6$ for simplicity:
\begin{congruence}
\label{cz2}
\[
b_{p-1}\equiv 1+\lp p^3-\frac{p^5}{2}\rp B_{p^3-p^2-2}-\frac{16}{5}p^5 B_{p-5}\mod p^6.
\]
\end{congruence}
\noindent Many other congruences can be derived with the same technique.

\begin{remark}
We can extend Congruence \ref{apez} to hold modulo arbitrary large powers of $p$ using $p$-adic analogues of the multiple zeta values. This is done by the author in \cite{Ros16b}.
\end{remark}

%%%%%%%%%%%
%
%   Mixed congruences
%
%%%%%%%%%%%
\section{Mixed congruences}
\label{secmix}
The previous section involves weighted multiple harmonic sums $\hp(\bs):=p^{|\bs|}\H(\bs)$.
One sometimes encounters expressions $p^b\H(\bs)$ with $b\neq|\bs|$. We make the following definition.
\begin{definition}
A \emph{mixed MHS congruence} (or just \emph{mixed congruence}) is a result of the form
\begin{equation}
\label{eqinh}
\sum_{i=1}^m a_i p^{b_i}\H(\bs_i)\equiv 0\mod p^n\gap\text{for all but finitely many $p$},
\end{equation}
where $a_1,\ldots, a_m\in\Q$, $b_1,\ldots,b_m\in\Z$, and $\bs_1,\ldots,\bs_m$ are compositions, all of which are independent of $p$.
\end{definition}

Thankfully, we expect that every mixed congruence is a consequence of weighted congruences. To be precise:

\begin{conjecture}
\label{conmix}
The mixed congruence \eqref{eqinh} holds if and only if for every $k\in\Z$, the weighted congruence
\[
\sum_{\substack{i=1\\|\bs_i|=b_i+k}}^{m}a_i\hp(\bs_i)\equiv 0\mod p^{n+k}
\]
holds for all $p$ sufficiently large.
\end{conjecture}
\noindent Note that the ``if'' part of the conjecture is obvious.

Sometimes a \ppc{} can be reduced to a mixed MHS congruence. The remainder of this section is devoted to three examples.

\subsection{A sum involving harmonic numbers}
\label{subsecsun}
In a work from 2012, Z.-W.\ Sun \cite{Sun12} proved four congruences for sums involving harmonic numbers. We consider one of them here:
\begin{equation}
\label{eqsun}
\sum_{k=1}^{p-1} H_k(1)^2\equiv 2p-2\mod p^2.
\end{equation}
The summand $H_k(1)^2$ can be expanded as $2 H_k(1,1)+H_k(2)$, so we have
\begin{align}
\sum_{k=1}^{p-1} H_k(1)^2&=\sum_{p-1\geq k\geq n>m\geq 1}\frac{2}{nm}+\sum_{p-1\geq k\geq n\geq 1}\frac{1}{n^2}\label{eqsu}\\
&=\sum_{p-1\geq n>m\geq 1}\frac{2(p-n)}{nm}+\sum_{p-1\geq n\geq 1}\frac{p-n}{n^2}\\
&=2p\H(1,1)-2\sum_{p-1\geq n>m\geq 1}\frac{1}{m}+p\H(2)-\H(1)\\
&=2p\H(1,1)-2\sum_{p-1\geq m\geq 1}\frac{p-1-m}{m}+p\H(2)-\H(1)\\
&=2p\H(1,1)-(2p-2)\H(1)+2(p-1)+p\H(2)-\H(1)\\
&=2p-2+(3-2p)\H(1)+p\H(2)+2p\H(1,1).
\end{align}
This shows that \eqref{eqsun} is equivalent to a mixed congruence. It also makes it easy to find strengthenings of \eqref{eqsun}. For example:
\begin{congruence}
\label{cs1}
\label{propss}
For all primes $p\geq 5$,
\[
\sum_{k=1}^{p-1} H_k(1)^2\equiv 2p-2+\frac{1}{3}p^2(2p-1)\H(2,1)\mod p^4.
\]
\end{congruence}
\begin{proof}
This is equivalent to the mixed MHS congruence
\begin{equation}
\label{eqmix}
(3-2p)\H(1)+p\H(2)+2p\H(1,1)\equiv  \frac{1}{3}p^2(2p-1)\H(2,1)\mod p^4.
\end{equation}
Equation \eqref{eqmix} follows from weighted congruences
\begin{gather*}
3\hp(1)+\hp(2)+2\hp(1,1)\equiv -\frac{1}{3}\hp(2,1)\mod p^5,\\
-2\hp(1)\equiv \frac{2}{3}\hp(2,1)\mod p^4,
\end{gather*}
which can be derived using Remark \ref{fact}.
\end{proof}

It is worth emphasizing that the expansion \eqref{eqsu}, and the subsequent proof of Proposition \ref{propss}, can be computed completely mechanically. We can similarly extend another congruence from \cite{Sun12}:
\begin{congruence} 
\label{cs2}
For $p\geq 5$,
\[
\sum_{k=1}^{p-1}k^2 H_k(1)^2\equiv-\frac{4}{9}+\frac{79}{108}p-\frac{13}{36}p^2+\frac{1}{6}\H(1)\mod p^3.
\]
\end{congruence}

\subsection{$p$-restricted harmonic numbers}\label{sspres}
We prove the following.
\begin{congruence}
\label{cr1}
\label{propmix}
For $p\geq 7$,
\begin{equation}
\label{exmix}
\sum_{\substack{n=1\\p\nmid n}}^{p^2}\frac{1}{n}\equiv p^2\sum_{n=1}^{p-1}\frac{1}{n}\mod p^6.
\end{equation}
\end{congruence}
The quantity on the left hand side of \eqref{exmix} is called a \emph{$p$-restricted harmonic number} because we exclude terms whose index is divisible by $p$. The $p$-restriction is just for convenience, and similar results are possible for the ordinary (non-restricted) harmonic numbers. We begin with a lemma.
\begin{lemma}
\label{lemp}
For all primes $p$, there is a convergent $p$-adic series identity
\begin{equation}
\label{eqpr}
\sum_{\substack{n=1\\p\nmid n}}^{p^2}\frac{1}{n}=\sum_{m=0}^\infty \sum_{k=1}^{m+1}(-1)^m{m+1\choose k}\frac{B_{m+1-k}}{m+1}p^{k+m}\H(m+1).
\end{equation}
\end{lemma}

\begin{proof}
We expand the left hand side as a $p$-adic series
\begin{align}
\sum_{\substack{n=1\\p\nmid n}}^{p^2}\frac{1}{n}&=\sum_{a=0}^{p-1}\sum_{j=1}^{p-1}\frac{1}{j+ap}\\
&=\sum_{a=0}^{p-1}\sum_{j=1}^{p-1}\frac{1}{j}\lp1+\frac{ap}{j}\rp^{-1}\\
&=\sum_{a=0}^{p-1}\sum_{j=1}^{p-1}\frac{1}{j}\sum_{m=0}^\infty(-1)^m\frac{a^m p^m}{j^m}\\
&=\sum_{m=0}^\infty(-1)^m p^m \H(m+1)\sum_{a=0}^{p-1}a^m\\
&=\sum_{m=0}^\infty(-1)^mp^{m}\H(m+1) \sum_{k=1}^{m+1}{m+1\choose k}\frac{B_{m+1-k}}{m+1}p^k.
\end{align}
In the last line we have used Faulhaber's formula for sums of $m$-th powers.
\end{proof}

\begin{proof}[Proof of Congruence \ref{propmix}]
We reduce \eqref{eqpr} modulo $p^6$ to find that \eqref{exmix} is equivalent to the mixed congruence
\begin{gather*}
p\H(1)+\lp\frac{p^2}{2}-\frac{p^3}{2}\rp\H(2)+\lp\frac{p^3}{6}-\frac{p^4}{2}+\frac{p^5}{3}\rp\H(3)\\-\frac{p^5}{4}\H(4)-\frac{p^5}{30}\H(5)\equiv p^2\H(1)\mod p^6.
\end{gather*}
This mixed congruence follows from the following weighted congruences, which can be checked algorithmically by Remark \ref{fact}:
\begin{align*}
\hp(1)+\frac{1}{2}\hp(2)+\frac{1}{6}\hp(3)-\frac{1}{30}\hp(5)\equiv 0\mod p^6,\\
-\frac{1}{2}\hp(2)-\frac{1}{2}\hp(3)-\frac{1}{4}\hp(4)\equiv \hp(1)\mod p^5,\\
\frac{1}{3}\hp(3)\equiv 0\mod p^4.
\end{align*}
\end{proof}
Like Proposition \ref{propss}, the proof of Proposition \ref{propmix} is  algorithmic. In Section \ref{ssmps} we generalize this technique to a wide range of multiple harmonic sums (and $p$-restricted variants). We give one additional congruence here, which was discovered and proved entirely by computer.
\begin{congruence}
\label{cr3}
\begin{gather*}
p^3H_{p^2-1}(2,1)\equiv \lp 1+p^3\rp\H(2,1)+\lp-\frac{11}{10}p^5+\frac{11}{10}p^7\rp\H(4,1)
\\\hspace{40mm}+\frac{7}{5}p^6\H(4,1,1)-\frac{59}{560}p^7\H(6,1)\mod p^8.
\end{gather*}
\end{congruence}

\subsection{The `curious' congruence}
\label{sscur}
Let $k$ and $r$ be positive integers, $p$ a prime. Congruence properties of the sum
\begin{equation}
\label{eqcur}
C_{r,k,p}:=\sum_{\substack{n_1,\ldots,n_k\geq 1\\n_1+\ldots+n_k=p^r\\p\nmid n_1\cdots n_k}}\frac{1}{n_1\cdots n_k}.
\end{equation}
have received attention from several authors, including Zhao \cite{Zha07,Zha14}, Ji \cite{Ji05}, Zhou and Cai \cite{Zho07}, Wang and Cai \cite{Wan14}, Wang \cite{Wan15}.
For $r=1$, it can be shown that
\[
C_{1,k,p}=\frac{k!}{p}\H(1^{k-1}),
\]
from which we can read off congruences in terms of Bernoulli numbers:

\[
C_{1,k,p}\equiv\begin{cases}-(k-1)! B_{p-k}\mod p\text{ if $k$ is odd},\\-\frac{n}{2(k+1)}k!B_{p-k-1}p\mod p^2\text{ if $k$ is even.}\end{cases}
\]
Some results are known for arbitrary $r$ and small $k$:
\begin{align*}
C_{r,2,p}&\equiv -\frac{2}{3}p^rB_{p-3}\mod p^{r+1}\gap\text{\cite{Zha14}},\\
C_{r,3,p}&\equiv -2\, p^{r-1}B_{p-3}\mod p^r\gap\text{\cite{Wan14}},\\
C_{r,4,p}&\equiv-\frac{4!}{5}p^r B_{p-5}\mod p^{r+1}\gap\text{\cite{Zha14}},\\
C_{r,5,p}&\equiv -\frac{5!}{6}p^{r-1}B_{p-5}\mod p^r\gap\text{\cite{Wan15}},\\
C_{r,6,p}&\equiv -\frac{5!}{18} p^{r-1}B_{p-3}^2\mod p^r\gap\text{\cite{Wan16}}.
\end{align*}

\begin{theorem}
\label{thbc}
Fix positive integers $r$, $k$, and $n$. Then there exist algorithmically computable compositions $\bs_1,\ldots,\bs_m$, coefficients $c_1,\ldots,c_m\in\Q$, and exponents $b_1,\ldots,b_m\in\Z$, independent of $p$, such that the congruence
\begin{equation}
\label{congC}
C_{r,k,p}\equiv \sum_{i=1}^m c_i p^{b_i}\H(\bs_i)\mod p^n
\end{equation}
holds for all sufficiently large $p$.
%In fact these congruence arise from a $p$-adically convergent infinite series
%\[
%C_{r,k,p}=\sum_{\bs}\hat{f}_{\bs}(p)\H(\bs),
%\]
%where $\hat{f}_{\bs}(x)\in\Q[[x]]$.
\end{theorem}
The proof is cumbersome, and can be found in Section \ref{teccur}.
The expansion \eqref{congC} is messy to work out by hand, but can be written down with the assistance of a computer. We give a few examples.
\begin{congruence}
\label{cc}
\[
C_{2,3,p}\equiv -2p\H(2,1)+\lp2p^3-\frac{11}{5}p^5\rp\H(4,1)-\frac{69}{35}p^5\H(6,1)\mod p^6,
\]

\[
C_{2,4,p}\equiv -\frac{24}{5}p^2\H(4,1)+\frac{28}{15}p^3\H(4,1,1)\mod p^4,
\]

\[
C_{3,3,p}\equiv -2 p^2\H(2,1)+2p^4\H(4,1)\mod p^6,
\]

\[
C_{3,4,p}\equiv -\frac{24}{5}p^3\H(4,1)+\frac{28}{15}p^4\H(4,1,1)\mod p^5.
\]
\end{congruence}

\section{Summary of the technique}
\label{ssmix}
Here we summarize our technique for proving \ppc{}s.

Suppose we encounter a quantity $a_p$ (depending on $p$) in a \ppc{}. We attempt to express $a_p$ as a $p$-adic series of the form
\begin{equation}
\label{eqser2}
a_p=\sum_{i=1}^{\infty}c_i p^{b_i}\H(\bs_i),
\end{equation}
where $c_i\in\Q$, the $\bs_i$ are composition, and $b_i\in\Z$ satisfy $b_i\to\infty$ as $i\to\infty$. We need the convergence of \eqref{eqser2} to be uniform in $p$ in the following weak sense: for each integer $n\geq 0$, the congruence
\begin{equation}
\label{eee}
a_p\equiv \sum_{\substack{i\\ b_i<n}}c_i p^{b_i}\H(\bs_i)
\end{equation}
should hold for all but finitely many $p$. In practice this uniformity condition is very often satisfied. In Sections \ref{secred} and \ref{secmix}, we described expansions of the form \eqref{eqser2} for many quantities: binomial coefficients, \ape{} numbers, values of the $p$-adic zeta function, various generalized harmonic numbers, various sums involving harmonic numbers, among others. In Section \ref{secred}, the expansions had the addition property that  $b_i=|\bs_i|$ for all $i$, i.e.\ they involved weighted multiple harmonic sums.

If we can find an expansion \eqref{eqser2} for every term in a \ppc{}, then that congruence is equivalent to a mixed MHS congruence, i.e., a congruence of the form \eqref{eqinh}. We decompose this mixed congruence into weighted parts, and attempt to prove each weighted part separately, using Remark \ref{fact}. The truth of Conjectures \ref{conmix} and \ref{conjds} would imply that if we fail to find a proof this way, then in fact the supercongruence fails for infinitely many primes.

%%%%%%%%%%%
%
%   Algebraic framework
%
%%%%%%%%%%%
\section{The MHS algebra}
\label{secalg}
In this section, we describe the class of quantities for which our technique applies.

\subsection{The ring $\Ai$}
In \cite{Ros13}, the author defined a complete topological ring related to the finite adeles, and used it to study weighted MHS congruences. Here we consider this ring with a prime $p$ inverted. This localization was considered in \cite{Jar16}, \S 6.4.2b, and we adopt the notation of \cite{Jar16}.

\begin{definition}
Define a commutative $\Q$-algebra
\[
\Ai:=\frac{\left\{(a_p)\in\prod_p\Q_p:v_p(a_p)\text{ bounded below}\right\}}{\left\{(a_p)\in\prod_p\Q_p:v_p(a_p)\to\infty \text{ as }p\to\infty\right\}}.
\]
We equip $\Ai$ with a decreasing, exhaustive, separated filtration:
\[
\fil^n\Ai=\left\{(a_p)\in\Ai: a_p\in p^n\Z_p\text{ for all but finitely many $p$}\right\}.
\]
The sets in the filtration form a neighborhood basis of $0$ for a topology, making $\Ai$ into a complete topological ring.
\end{definition}

%That $\Ai$ is complete follows frome the fact that it is the quotient of the complete, first countable ring $\prod_p\Q_p$ with the uniform topology (i.e., the sets $\prod_p p^n\Z_p$ are a neighborhood basis of $0$) modulo a closed ideal. The notion of convergence in $\Ai$ is surpising at first, as it is distinct from $p$-adic convergence for all but finitely many $p$. A sequence $(a_{p,1})$, $(a_{p,2})$, $\ldots$ converges to $(b_p)\in\Ai$ if and only if for every positive integer $m$, there exists $N=N(m)$ such that the congruence
%\[
%a_{p,n}\equiv b_p\mod p^m\Z_p\gap\text{for all but finitely many $p$}
%\]
%holds whenever $n>N$. The finite set of primes for which the congruence fails may depend on $n$, so convergence in $\Ai$ does \emph{not} imply that $a_{p,n}\to b_p$ $p$-adically for any $p$ at all. An example is
%\[
%a_{p,n}=\begin{cases}1:p\leq n,\\0:p>n.\end{cases}
%\]
%Then $a_{p,n}\to 1$ for every $p$, but $(a_{p,n})\to 0$ in $\Ai$, and in fact $(a_{p,n})$ is the constant sequence $0$.

\subsection{The MHS algebra}

\begin{definition}
\label{defmhsalg}
The \emph{mhs algebra} is the subset $\ma\subset\Ai$ consisting of elements $(a_p)$ such that there exist rational numbers $c_1,c_2,\ldots$, integers $b_1,b_2,\ldots$ going to infinity, and compositions $\bs_1,\bs_2,\ldots$, all independent of $p$, such that
\begin{equation}
\label{eqsumA}
(a_p)=\sum_{i=1}^\infty \lp c_i p^{b_i}\H(\bs_i)\rp\in\Ai.
\end{equation}
Concretely, this is equivalent to the condition that for every integer $n$, the congruence
\begin{equation}
\label{eqss}
a_p\equiv \sum_{\substack{i=1\\b_i<n}}^\infty c_i p^{b_i}\H(\bs_i)\mod p^n
\end{equation}
holds for all sufficiently large $p$ (note that the sum of the right hand side of \eqref{eqss} is finite).
\end{definition}
 Sometimes we will abuse terminology and write that the quantity $a_p$ is in the MHS algebra when we mean $(a_p)\in\mhs$.  If $(a_p)$, $(b_p)\in\mhs$, then the \ppc{}
\[
a_p\equiv b_p\mod p^n
\]
is equivalent to a mixed MHS congruence, hence by Remark \ref{fact} has a conjectural decision procedure which can be performed by a computer.

In Section \ref{secred}, we found expansions of the form \eqref{eqsumA} with the additional property that $b_i=|\bs_i|$ for all $i$ (that is, expansions in terms of \emph{weighted} multiple harmonic sums).
\begin{definition}
The \emph{weighted MHS algebra} is the set of elements $(a_p)$ in the MHS algebra such that the $c_i,b_i,\bs_i$ in Definition \ref{defmhsalg} can be chosen with $b_i=|\bs_i|$ for all $i$.
\end{definition}

Elements of the weighted MHS algebra were called \emph{asymptotically representable by weighted multiple harmonic sums} in \cite{Ros13}.

\section{More elements of the MHS algebra}
\label{sectec}
In this section, we show that a number of additional quantities that have appeared in \ppc{}s are in the MHS algebra. All of the results here are constructive, in that the proof gives a recipe for expressing a quantity as a sum of the form \eqref{eqsumA}.

\subsection{Generalities}

\begin{proposition}[\cite{Ros13}, Lemma 6.2]
The weighted MHS algebra is topologically closed.
\end{proposition}
%\begin{proof}
%This follows from the fact that the weighted MHS algebra is the image of the pro-Artinian ring $\qsh$ under a continuous ring homomorphism (see \cite{Ros13}, Lemma 6.2).
%\end{proof}
The usefulness of this fact is that if we can express an element $(a_p)\in\Ai$ as a convergent infinite sum of elements in the weighted MHS algebra, then we can conclude $(a_p)$ is also in the weighted MHS algebra. 

Unfortunately, the proof does not generalize to the MHS algebra. We compensate by making the following definition.
\begin{definition}
We define a filtration $\mf^n$ on the MHS algebra $\ma$ by taking $\mf^n\ma$ to be the set of $(a_p)\in\ma$ such that there is an expansion \eqref{eqsumA} satisfying $b_i\geq n$ for all $i$.
\end{definition}

The truth of Conjectures\ \ref{conmix} and \ref{conjds} would imply that $\mf^n\ma=\fil^n\Ai\cap\ma$, and therefore would also imply that $\ma$ is closed.

\begin{proposition}
Let $\alpha_1,\alpha_2,\ldots$ be a sequence of elements of $\ma$ such that for all $N\in\Z$, all but finitely many of the terms of the sequence are in $\mf^N\ma$. Then then infinite sum
\begin{equation}
\label{eqsum}
\sum_{n=1}^\infty\alpha_n
\end{equation}
converges to an element of $\ma$.
\end{proposition}
\begin{proof}
Choose expansions \eqref{eqsumA} for $\alpha_1,\alpha_2,\ldots$ such that the smallest $b_i$ appearing in $\alpha_n$ goes to $\infty$ as $n\to\infty$, and formally add the expansions.
\end{proof}

\begin{proposition}
\label{propunit}
If $a\in\mf^1\ma$, then $1+a$ is a unit in the MHS algebra.
\end{proposition}
\begin{proof}
The inverse of $1+a$ is $\sum_{n\geq 0} (-a)^n$.
\end{proof}

\subsection{Rational functions}
\begin{proposition}
\label{proprat}
Let $f(x)\in\Q(x)$ be a rational function. Then $f(p)$ is in the MHS algebra.
\end{proposition}
\begin{proof}
There is a Laurent series expansion
\[
f(x)=\sum_{i=-N}^{\infty}b_i x^i\in\Q((x)).
\]
It is known that for all but finitely many $p$, every coefficient is $p$-integral. For those $p$, the congruence
\[
f(p)\equiv\sum_{i=-N}^{n-1}b_i p^i\mod p^n
\]
holds for all positive integers $n$.
\end{proof}

\subsection{Multiple power sums}
\label{ssmps}
For non-negative integers $N$, $M$, arbitrary integers $s_1,\ldots,s_k$, and $p$ a prime, define
\begin{gather*}
S_{N,M}(s_1,\ldots,s_k):=\sum_{N\geq n_1>\ldots>n_k\geq M+1}\frac{1}{n_1^{s_1}\ldots n_k^{s_k}},\\
S^{(p)}_{N,M}(s_1,\ldots,s_k):=\sum_{\substack{N\geq n_1>\ldots>n_k\geq M+1\\p\nmid n_1\ldots n_k}}\frac{1}{n_1^{s_1}\ldots n_k^{s_k}}.
\end{gather*}
When $s_1,\ldots,s_k$ are positive, we have $S_{N,0}(s_1,\ldots,s_k)=H_N(s_1,\ldots,s_k)$. We prove the following result.

\begin{theorem}
\label{thmps}
Let $s_1,\ldots,s_k$ be arbitrary integers, and suppose $f(x)$, $g(x)\in\Z[x]$ have positive leading coefficients. Then
\begin{gather*}
S_{f(p),g(p)}(s_1,\ldots,s_k) \in\mf^{-\deg(g)\big([s_1]^++\ldots+[s_k]^+\big)}\ma,\\  S^{(p)}_{f(p),g(p)}(s_1,\ldots,s_k)\in\mf^0\ma.
\end{gather*}
Here, $[n]^+$ denotes the positive part of $n$, which is $n$ if $n$ is positive and $0$ otherwise.
\end{theorem}

\begin{corollary}
For every polynomial $f(x)\in\Z[x]$ with positive leading coefficient and every composition $\bs$, the multiple harmonic sum $H_{f(p)}(\bs)$ is in the MHS algebra.
\end{corollary}

We begin with some lemmas.
%\begin{lemma}
%\label{pp}
%For fixed integers $s_1,\ldots,s_k$, the function $\Z_{\geq0}\to \Q$,
%\begin{equation}
%\label{eqS}
%N\mapsto S_{N,0}(s_1,\ldots,s_k)
%\end{equation}
%can be written as an explicitly computable finite linear combination with rational coefficients of functions of the form
%\[
%N\mapsto N^a H_N(\bt),
%\]
%with $a\in\Z_{\geq 0}$ and $\bt$ a composition.
%\end{lemma}
%
%\begin{proof}
%The proof is by induction on $k$ (the number of $s_i$). If all the $s_i$ are positive, $S_{N,0}(s_1,\ldots,s_k)=H_N(\bs)$ and we are done. If some $s_i$ is non-positive, then the quantity
%\[
%Q(n_{i-1},n_{i+1}):=\sum_{n_{i-1}>n_i>n_{i+1}} \frac{1}{n_i^{s_i}}
%\]
%is a polynomial in $n_{i-1}$ and $n_{i+1}$. We can thus reduce $k$ by $1$.
%\end{proof}

\begin{lemma}
\label{lemaj}
For all integers $b$, $r$, $s_1,\ldots,s_k$ with $b\geq 1$ and $r\geq 0$,
\begin{gather}
\label{eqthing}
S_{bp^r,(b-1)p^r}(s_1,\ldots,s_k)\in\mf^{-r\big([s_1]^++\ldots+[s_k]^+\big)}\ma,\\
S^{(p)}_{bp^r,(b-1)p^r}(s_1,\ldots,s_k)\in\mf^{0}\ma.
\end{gather}             
\end{lemma}

A closely related result appeared in \cite{Jar16}, \S 8.2.
\begin{proof}
We give the proof for $S_{bp^r,(b-1)p^r}$. The argument for $S_{bp^r,(b-1)p^r}^{(p)}$ is similar (and in fact is easier).

The proof is by induction on $r$. For $r=0$, the quantity is a constant rational number not depending on $p$, so is in $\mf^0\ma$.

Now assume the result holds up through $r-1$. In the sum defining $S_{bp^r,(b-1)p^r}$, we substitute $n_i=a_ip-j_i$ for $i=1,\ldots,k$ to obtain
\begin{equation}
\label{niceS}
S_{bp^r,(b-1)p^r}(s_1,\ldots,s_k)=\sum_{a_i,j_i}\frac{1}{(a_1p-j_1)^{s_1}\ldots (a_kp-j_k)^{s_k}},
\end{equation}
where the sum of over $a_i$, $j_i$ satisfying
\begin{itemize}
\item $bp^{r-1}\geq a_1\geq \ldots\geq a_{k-1}\geq (b-1)p^{r-1}+1$,
\item $p-1\geq j_1,\ldots,j_k\geq 0$, and
\item for each $i$, if $a_i=a_{i+1}$ then $j_i<j_{i+1}$.
\end{itemize}

We separate the sum \eqref{niceS} into finitely many sub-sums depending on the possible (not necessarily strict) orderings of $a_1,\ldots,a_k$ and $0,j_1,\ldots,j_j$, and we get that $S_{bp^r,(b-1)p^r}(s_1,\ldots,s_k)$ can be written as a finite sum of terms of the form
\begin{equation}
\label{subsum}
\sum_{b p^{r-1}\geq a_1>\ldots>a_n\geq (b-1)p^{r-1}+1}\hspace{-15mm}\big[(a_1 p)^{c_1}\ldots (a_np)^{c_n}\big]^{-1}\hspace{-3mm}\sum_{p-1\geq j_1>\ldots>j_m\geq 1\hspace{-5mm}}(a_{e_1}p-j_{f_1})^{d_1}\ldots (a_{e_m}p-j_{f_m})^{d_m},
\end{equation}
where
\begin{itemize}
\item $n$, $m\in\Z_{\geq 0}$,
\item $c_1,\ldots,c_n\in\Z$, with $\sum [c_i]^+ \leq \sum[s_i]^+$
\item $e_1,\ldots,e_m$ and $f_1,\ldots,f_m$ are in $\{1,\ldots,n\}$, and
\item $d_1,\ldots,d_m\in\Z$.
\end{itemize}
Now for $1\leq j\leq p-1$, there is a uniformly convergent $p$-adic series expansion
\[
(ap-j)^{-s}=\sum_{n\geq 0}{-s\choose n} (-1)^{s+n}a^np^nj^{-s-n}.
\]
Thus we can write $S_{bp^r,(b-1)p^r}(s_1,\ldots,s_k)$ as a $p$-adically uniformly convergent sum of terms
\begin{equation}
\label{factor}
p^N\hp(t_1,\ldots,t_w) S_{bp^{r-1},(b-1)p^{r-1}}(u_1,\ldots,u_q),
\end{equation}
with $N\geq -\sum[s_i]^+$ and $\sum [u_i]^+\leq \sum[s_i]^+$. By the inductive hypothesis, the rightmost factor in \eqref{factor} is in $\mf^{\sum[u_i]^+}\ma$, so we conclude
\[
S_{bp^{r},(b-1)p^{r}}(\bs)\in p^N\mf^{-(r-1)\sum[u_i]^+}\ma\subset\mf^{-r\sum[s_i]^+}\ma.
\]
\end{proof}
\begin{lemma}
\label{lemmm}
For all integers $b$, $r$, $s_1,\ldots,s_k$, with $b,r\geq 1$:
\begin{gather}
\label{eqqq}
S_{bp^r,0}(s_1,\ldots,s_k)\in\mf^{-r\big([s_1]^++\ldots+[s_k]^+\big)}\ma,\\
S^{(p)}_{bp^r,0}(s_1,\ldots,s_k)\in\mf^{0}\ma.
\end{gather}
\end{lemma}
\begin{proof}
There are finitely many ways the summands in the definition of \eqref{eqqq} can be divided among the intervals $[1,p^r]$, $[p^r+1,2p^r]$, $\ldots$, $[(b-1)p^r+1,bp^r]$. In this way \eqref{eqqq} can be written as finite sums of products of terms of the form \eqref{eqthing}.
\end{proof}

\begin{lemma}
\label{lem0}
Let $f(x)\in\Z[x]$ be a polynomial with positive leading coefficient, $s_1,\ldots,s_k\in\Z$. Then
\begin{gather}
\label{eqff}
S_{f(p),0}(s_1,\ldots,s_k)\in\mf^{-\deg(f)\big([s_1]^++\ldots+[s_k]^+\big)}\ma,\\
S^{(p)}_{f(p),0}(s_1,\ldots,s_k)\in\mf^{0}\ma.
\end{gather}

\end{lemma}
\begin{proof}
We prove this for $S_{f(p),0}(s_1,\ldots,s_k)$. The other case is similar.

We induct on the degree of $f$. If $\deg(f)=0$, the quantity $S_{f(p),0}(s_1,\ldots,s_k)$ is a rational constant independent of $p$, so is in the MHS algebra.

Assume $\deg(f)=r\geq 1$. We can write $f(x) = ax^r +g(x)$ or $f(x)=ax^r -g(x)$ where $g(x)$ is either $0$ or a polynomial of degree less than $r$ with positive leading coefficient.

We will assume that $f(x) = ax^r +g(x)$ as above (the case $f(x)=ax^r -g(x)$ is similar). We have
\[
S_{f(p),0}(s_1,\ldots,s_k)=\sum_{i=0}^kS_{ap^r+g(p),ap^r}(s_1,\ldots,s_i)S_{ap^r,0}(s_{i+1},\ldots,s_k).
\]
So it suffices to show that
\[
S_{ap^r+g(p),ap^r}(s_1,\ldots,s_i)\in\mf^{-r\big([s_1]^++\ldots+[s_i]^+\big)}\ma.
\]

For all sufficiently large $p$, we will have $ap^r>g(p)$. For such $p$,
\begin{gather*}
S_{ap^r+g(p),ap^r}(s_1,\ldots,s_i)\\
=\sum_{g(p)\geq n_1>\ldots>n_i\geq 1}\frac{1}{(ap^r+n_1)^{s_1}\ldots(ap^r+n_i)^{s_i}}\\
=\sum_{g(p)\geq n_1>\ldots>n_i\geq 1}\frac{\lp1+\frac{ap^r}{n_1}\rp^{-s_1}\ldots \lp1+\frac{ap^r}{n_i}\rp^{-s_i}}{n_1^{s_1}\ldots n_i^{s_i}}\\
=\sum_{t_1,\ldots,t_i\geq 0}\lp\prod_{j=1}^i{-s_j\choose t_j}\rp (ap^r)^{\sigma t}S_{g(p),0}(s_1+t_1,\ldots,s_i+t_i).
\end{gather*}
The result now follows by applying the inductive hypothesis to the term $S_{g(p),0}(s_1+t_1,\ldots,s_i+t_i)$.
\end{proof}
%As an immediate consequence, we get that for any polynomial $f(x)\in\Z[x]$ with positive leading coefficient,
%\[
%H_{f(p)}(s_1,\ldots,s_k),\gap H^{(p)}_{f(p)}(s_1,\ldots,s_k)
%\]
%are in the MHS algebra.

\begin{proof}[Proof of Theorem \ref{thmps}]
The proof is induction on $k$. For $k=0$ the result is trivial. For the induction step, observe that
\begin{gather*}
\hspace{-50mm}S_{f(p),g(p)}(s_1,\ldots,s_k)=S_{f(p),0}(s_1,\ldots,s_k)\\\hspace{40mm}-\sum_{i=0}^{k-1} S_{f(p),g(p)}(s_1,\ldots,s_i)S_{g(p),0}(s_{i+1},\ldots,s_k)
\end{gather*}
and use Lemma \ref{lem0}. The proof for $S^{(p)}$ is identical.
\end{proof}

\subsection{Binomial coefficients}
\label{ssbin}
\begin{theorem}
\label{thgenbin}
Let $f(x),g(x)\in\Z[x]$ have positive leading coefficients. Then the binomial coefficient
\begin{equation}
\label{defbin}
{f(p)\choose g(p)}
\end{equation}
is a unit in the MHS algebra.
\end{theorem}

%
%We need the following computation.
%\begin{lemma}
%\label{lembinomial}
%For all non-integers $a$, $b$, $n$:
%\begin{gather*}
%{a+n\choose b}={a\choose b}\frac{\lp1+\frac{a}{1}\rp\lp1+\frac{a}{2}\rp\ldots\lp 1+\frac{a}{n}\rp}{\lp1+\frac{a-b}{1}\rp\lp1+\frac{a-b}{2}\rp\ldots\lp1+\frac{a-b}{n}\rp},\\
%{a-n\choose b}=(-1)^{n-1}{a\choose b}\frac{a-b-n}{a}\frac{\lp1-\frac{a-b}{1}\rp\lp1-\frac{a-b}{2}\rp\ldots\lp 1-\frac{a-b}{n-1}\rp}{\lp1+\frac{a}{1}\rp\lp1+\frac{a}{2}\rp\ldots\lp1+\frac{a}{n-1}\rp}\\
%{a\choose b+n}=(-1)^{n-1}{a\choose b}\frac{a-b}{n}\frac{\lp1-\frac{a-b}{1}\rp\lp1-\frac{a-b}{2}\rp\ldots\lp 1-\frac{a-b}{n-1}\rp}{\lp1+\frac{b}{1}\rp\lp1+\frac{b}{2}\rp\ldots\lp1+\frac{b}{n-1}\rp},\\
%{a\choose b-n}=(-1)^{n-1}{a\choose b}\frac{b}{a-b+n}\frac{\lp1-\frac{b}{1}\rp\lp1-\frac{b}{2}\rp\ldots\lp 1-\frac{b}{n-1}\rp}{\lp1+\frac{a-b}{1}\rp\lp1+\frac{a-b}{2}\rp\ldots\lp1+\frac{a-b}{n-1}\rp}.\\
%\end{gather*}
%\end{lemma}
%The proof is straightforward.

\begin{lemma}
\label{lempow}
For integers $a\geq b\geq 0$ and $r\geq 0$, the quantity
\[
{ap^r\choose bp^r}
\]
is a unit in the MHS algebra.
\end{lemma}
\begin{proof}
The proof is induction on $r$. For $r=0$ this is trivial.

For $p\geq 3$, we have
\begin{align*}
{ap^r\choose bp^r}&={ap^{r-1}\choose bp^{r-1}}\prod_{\substack{i=1\\p\nmid i}}^{bp^r}\lp 1-\frac{ap^r}{i}\rp\\
&={ap^{r-1}\choose bp^{r-1}}\cdot\sum_{n\geq 0}(-ap^r)^n H_{bp^r}^{(p)}(1^n).
\end{align*}
Now use the inductive hypothesis and Lemma \ref{lem0}, noting that the sum on the right is a unit by Proposition \ref{propunit}
\end{proof}

\begin{proof}[Proof of Theorem \ref{thgenbin}]
If $\deg(f)<\deg(g)$, then \eqref{defbin} is $0$ for $p$ sufficiently large, so we may assume $\deg(f)\geq\deg(g)$.

We induct on $\deg(f)$. If $\deg(f)=0$, then \eqref{defbin} is constant, so is in the MHS algebra.

Suppose $\deg(f)=r\geq 1$. We can write $f(x)=ax^r+j(x)$, $g(x)=bx^r+j(x) +k(x)$, with $j(x)$, $k(x)$ polynomials of degree at most $r-1$. Consider the case when the leading coefficients of $j(x)$ and of $k(x)$ are non-negative (the other cases are similar). For $p$ sufficiently large that $j(p)$, $k(p)\geq 0$, we compute
\begin{align*}
{f(p)\choose g(p)}&={ap^r+j(p)\choose bp^r+j(p)+k(p)}\\
&={ap^r\choose bp^r}\frac{\lp\prod_{i=1}^{j(p)}(ap^r+i)\rp(a-b)p^r\lp\prod_{i=1}^{k(p)-1}((a-b)p^r-i)\rp}{\prod_{i=1}^{j(p)+k(p)}(bp^r+i)(bp^r+2)\ldots(bp^r+j(p)+k(p)}\\
&={ap^r\choose bp^r}\cdot\frac{(a-b)p^r}{bp^r+j(p)+k(p)}\cdot(-1)^{k(p)-1}\cdot{j(p)+k(p)-1\choose j(p)}^{-1}\cdot\\
&\hspace{15mm}\lp\sum_{n\geq 0} a^np^{nr} H_{j(p)}(1^n)\rp\cdot \lp\sum_{n\geq 0} (b-a)^np^{nr} H_{k(p)-1}(1^n)\rp\cdot\\
&\hspace{15mm} \lp\sum_{n\geq 0} b^np^{nr} H_{j(p)+k(p)-1}(1^n)\rp^{-1}.
\end{align*}
In the last line, each factor is a unit in the MHS algebra: the first factor by Lemma \ref{lempow}, the second factor by Proposition \ref{proprat}, the third factor is a non-zero constant for $p$ odd, the fourth factor by inductive hypothesis, and the fifth, sixth, and seventh factors by Proposition \ref{propunit}.

\end{proof}

\subsection{Alternating harmonic numbers}
\begin{theorem}
\label{thalt}
For all integers $k\geq 2$, the quantities
\[
p^k\sum_{n=1}^{(p-1)/2}\frac{1}{n^k},\gap p^k\sum_{n=1}^{p-1}\frac{(-1)^n}{n^k}
\]
are in the weighted MHS algebra.
\end{theorem}

\begin{proof}
We will use three properties of the Bernoulli polynomials $B_k(x)$:
\begin{gather}
\label{eqb1}\sum_{j=0}^{n-1}j^k=\frac{B_{k+1}(n)-B_{k+1}}{k+1},\\
\label{eqb2}B_{k}\lp\frac{1}{2}\rp=(2^{1-k}-1)B_k,\\
\label{eqb3}B_k(x+y)=\sum_{j=0}^k {k\choose j}B_{k-j}(x)y^j.
\end{gather}
We compute
\begin{align*}
p^kH_{\frac{p-1}{2}}(k)&=p^k\lim_{m\to\infty}\sum_{n=1}^{(p-1)/2}n^{\varphi(p^m)-k}\\
&=\lim_{m\to\infty}\frac{B_{\varphi(p^m)-k+1}\lp\frac{p+1}{2}\rp-B_{\varphi(p^m)-k+1}}{\varphi(p^m)-k+1}\\
&=p^k\lim_{m\to\infty}\frac{\left[\sum_{j=0}^{\varphi(p^m)-k+1}{\varphi(p^m)-k+1\choose j}B_{\varphi(p^m)-k+1-j}\lp\frac{1}{2}\rp \lp\frac{p}{2}\rp^j \right]-B_{\varphi(p^m)-k+1}}{\varphi(p^m)-k+1}\\
&=\lim_{m\to\infty} \left[\sum_{j=0}^{\varphi(p^m)-k+1}{\varphi(p^m)-k+1\choose j}\frac{2^{k-\varphi(p^m)+j}-1}{2^j}\frac{B_{\varphi(p^m)-k+1-j}}{\varphi(p^m)-k+1}p^{k+j}\right]\\
&\hspace{30mm}-\frac{B_{\varphi(p^m)-k+1}}{\varphi(p^m)-k+1}\\
&=\zeta_p(k)+\sum_{j\geq 0}{-k\choose j}\frac{1-2^{k+j}}{2^j}\zeta_p(k+j),
\end{align*}
which proves the first assertion. The second assertion now follows from
\[
p^k\sum_{n=1}^{p-1}\frac{(-1)^n}{n^k}=\frac{1}{2^{k-1}}p^k H_{\frac{p-1}{2}}(k)-\hp(k).
\]
\end{proof}

As an easy application, we obtain the following:
\begin{congruence}
\label{congalt}
For all primes $p\geq 7$,
\begin{equation}
\label{eqalt}
\sum_{n=1}^{p-1}\frac{(-1)^n}{n^2}\equiv \frac{3}{4}\sum_{n=1}^{p-1}\frac{1}{n^2}\mod p^3.
\end{equation}
\end{congruence}
\begin{proof}
By Theorem \ref{thalt}, \eqref{eqalt} is equivalent to a weighted congruence, which can be checked algorithmically.
\end{proof}
There is a notable similarity between \eqref{eqalt} and the well-knon identity
\begin{equation}
\label{eqeuler}
\sum_{n=1}^{\infty}\frac{(-1)^n}{n^2}=-\frac{3}{4}\sum_{n=1}^\infty\frac{1}{n^2},
\end{equation}
although the right hand sides differ in sign.

\subsection{The curious congruence}
\label{teccur}
Recall that we define
\[
C_{r,k,p}:=\sum_{\substack{n_1+\ldots+n_k=p^r\\p\nmid n_1\ldots n_k}}\frac{1}{n_1\ldots n_k}.
\]
\begin{proposition}
For fixed $r$, $k\geq 1$, the quantity $C_{r,k,p}$ is in the MHS algebra.
\end{proposition}
\begin{proof}
We use the well-known identity
\[
\sum_{\sigma \in S_n}\frac{1}{x_{\sigma(1)}(x_{\sigma(1)}+x_{\sigma(2)})\ldots(x_{\sigma(1)}+x_{\sigma(2)}+\ldots+x_{\sigma(n)})}=\frac{1}{x_1\ldots x_n}.
\]
We apply this identity to $C_{r,k,p}$ and write $m_i=n_i+n_{i+1}+\ldots+n_k$ to get
\[
C_{r,k,p}=k!\sum_{\substack{p^r=m_1>m_2>\ldots>m_k\geq 1\\p\nmid m_k(m_{k-1}-m_k)\ldots(m_1-m_2)}}\frac{1}{m_1\ldots m_k}.
\]
From here the proof essentially proceeds as in the proof of Lemma \ref{lemaj}. We write $m_i=a_ip-j_i$ to obtain
\begin{equation}
\label{curbreak}
C_{r,k,p}=k!\sum_{a_i,j_i}\frac{1}{(a_1p-j_1)\ldots(a_kp-j_k)},
\end{equation}
where the sum is over $a_i$, $j_i$ satisfying
\begin{itemize}
\item $p^{r-1} = a_1\geq a_2\geq \ldots\geq a_k\geq 1$ and $p-1\geq j_1,\ldots,j_k\geq 0$,
\item if $a_i=a_{i+1}$ then $j_i<j_{i+1}$, and
\item $j_1=0$, $j_k\neq 0$, and $j_i\neq j_{i+1}$.
\end{itemize}
We separate \eqref{curbreak} into finitely many sub-sums of the form \eqref{subsum} depending on the possible orderings of the $a_i$ and $j_i$. In the proof of Lemma \ref{lemaj}, we showed that these subsums are in the MHS algebra.
\end{proof}

%
%\subsection{$p$-adic \mzv{}s}
%
%\todo{Rewrite this section}
%
%There are $p$
%
%Jarrosay \cite{Jar16} considers a family of $p$-adic multiple zeta values $\zeta_{p,-k}$, for $k$ a positive integer. The case $k=1$ are the values constructed by Deligne.
%\begin{proposition}
%For any composition $\bs=(s_1,\ldots,s_n)$ with $s_1\geq 1$ and any positive integer $k$, Jarossay's $p$-adic \mzv{} $\zeta_{p,-k}(\bs)$ is in the MHS algebra.
%\end{proposition}
%\begin{proof}
%Proposition  of \cite{Jar16} implies $\zeta_{p,-k}(\bs)$ can be expressed as an infinite $\Q$-linear combination of terms
%\[
%p^{k|\bt|}H_{p^r-1}(\bt),
%\]
%as $\bt$ ranges over compositions. Combining this with Theorem \ref{thmps} completes the proof.
%\end{proof}
%
%Furusho \cite{Fur04, Fur07} defines a different $p$-adic \mzv{}s. An explicit formula relating Furusho's $p$-adic \mzv{}s to those considered here is given in \cite{Fur07}, Theorem 2.8. An immediate consequence of this formula is the following.
%
%\begin{proposition}
%Furusho's $p$-adic \mzv{}s are in the MHS algebra.
%\end{proposition}

%%%%%%%%%%%
%
%   Twisted MHS
%
%%%%%%%%%%%
\section{Multiple polylogarithms and truncations}
\label{secmpl}
Many of the results of this paper can be applied to a more general class of finite sums.
\begin{definition}
Suppose $s_1,\ldots,s_k$ are positive integers and $z_1,\ldots,z_k\in\Q$. We define
\begin{equation}
\label{tmpl}
H_N(s_1,\ldots,s_k; z_1,\ldots,z_k):=\sum_{N\geq n_1>\ldots>n_k\geq 1}\frac{z_1^{n_1}\ldots z_k^{n_k}}{n_1^{s_1}\ldots n_k^{s_k}}\in\Q.
\end{equation}
\end{definition}
Much of the present work can be generalized to include sums \eqref{tmpl} (though known methods may no longer be sufficient to determine all supercongruences).  We will not work out the general theory here, but content ourselves with expanding a few quantities as convergent series involving sums \eqref{tmpl}.

\subsection{Alternating multiple harmonic sums}
There is particular interest in the situation $z_1,\ldots,z_k\in\{\pm 1\}$. In this case, the sum \eqref{tmpl} is called an \emph{alternating multiple harmonic sum}. There is a special notation for such sums: the $z_i$ are omitted, and we write a bar over $s_i$ if $z_i=-1$. So, for example,
\[
H_N(1,\overline{2},\overline{3})=H_N(1,2,3;1,-1,-1).
\]

\begin{theorem}
For every prime $p\geq 3$, we have
\begin{equation}
\label{2top}
2^p = 2+\sum_{k\geq 0}(-1)^{k+1} p^{k+1}\H(\overline{1},1^k).
\end{equation}
Here we write
\[
\H(\overline{1},1^k)=\H(\overline{1},\underbrace{1,\ldots,1}_k).
\]
\end{theorem}

The reduction of \eqref{2top} modulo $p$ is Euler's congruence $2^p\equiv 2\mod p$. The reduction of \eqref{2top} modulo $p^2$ is the result
\begin{equation}
\label{eq2}
2^p\equiv 2-p\sum_{n=1}^{p-1}\frac{(-1)^n}{n}\mod p^2,
\end{equation}
which is also well-known. The identity \eqref{2top} shows that \eqref{eq2} can be extended to congruences holding modulo arbitrarily large powers of $p$, by introducing additional alternating multiple harmonic sums.

The expression \eqref{2top} can be generalized.
\begin{theorem}
For every positive integer $k$ and every prime $p$, we have
\[
n^p=n+\sum_{k\geq 0}(-1)^n (-1)^{k+1}\sum_{a=1}^{n-1}p^{k+1}\H(1^{k+1};-a,1^k).
\]
\end{theorem}
\begin{proof}
The proof is induction on $n$. For $n=1$ this is trivial as the sum on the right is empty. The inductive step follows from the expansion
\begin{align*}
n^p&=(1+(n-1))^p\\
&=1 + (n-1)^p+\sum_{j=1}^{p-1}{p\choose j}(n-1)^j\\
&=1+(n-1)^p+\sum_{j=1}^{p-1}(-1)^{j-1}\frac{p}{j}\lp1-\frac{p}{1}\rp\lp1-\frac{p}{2}\rp\ldots\lp1-\frac{p}{j-1}\rp(n-1)^j\\
&=1+(n-1)^p+\sum_{j=1}^{p-1}(-1)^{j-1}\frac{p(n-1)^j}{j}\sum_{k\geq 0}(-1)^kp^k H_{j-1}(1^k)\\
&=1+(n-1)^p+\sum_{k\geq 0}(-1)^{k+1}p^{k+1}\H(1^{k+1};1-n,1^k).
\end{align*}
\end{proof}

\appendix

\section{Software}
\label{appsoft}
We provide software for performing computations described in this paper (finding expansions for quantities in terms of multiple harmonic sums, computing weighted and mixed congruences, finding and proving \ppc{}s, etc.). The software is available at
\begin{center}
\url{https://sites.google.com/site/julianrosen/mhs}.
\end{center}
The software is written in Python 2.7, and it uses the SymPy library, as well as computations from the MZV data mine \cite{Blu10}.

Here we give a few example computations. Further documentation is included with the software itself.

\subsection{Valuation}
The method \verb|.v()| computes a lower bound on $p$-adic valuation that the software can prove. For instance, we can verify Congruence \ref{cb}:

\begin{verbatim}
>>> (12 - 9*binp(2,1) + 2*binp(3,1) - 24*hp(3)).v()
6
\end{verbatim}

\subsection{Expressing a quantity in terms of multiple harmonic sums}
The method \verb|.mhs()| expresses a quantity in terms of multiple harmonic sums. For example, we can obtain the third relation of Congruence \ref{cc}:
\begin{verbatim}
>>> a=curp(3,3); a.disp()
\end{verbatim}
$\displaystyle \sum_{\substack{n_1 + n_2 + n_3 = p^3\\p\nmid n_{1}n_{2}n_{3}}}\frac{1}{n_1n_2n_3}$

\begin{verbatim}
>>> a.mhs()
\end{verbatim}
\begin{flalign*}
&-2p^2\H(2, 1) + 2p^4\H(4, 1) + O(p^5).&
\end{flalign*}
\begin{verbatim}
>>> b=aperybp(); b.disp()
\end{verbatim}
$\displaystyle \sum_{n=0}^{p-1}{p-1\choose n}^2{p+n-1\choose n}^2$

\begin{verbatim}
>>> b.mhs()
\end{verbatim}
\begin{gather*}
1 + \frac{2}{3}p^{3}H_{p-1}(2, 1) -  \frac{59}{15}p^{5}H_{p-1}(4, 1) -  \frac{22}{45}p^{6}H_{p-1}(4, 1, 1) -  \frac{11953}{2520}p^{7}H_{p-1}(6, 1)\\
 + p^{8}\left(\frac{110321}{2700}H_{p-1}(6, 1, 1) + \frac{480701}{37800}H_{p-1}(5, 2, 1)\right) + O(p^{9}).
\end{gather*}

\section{Decision procedure for weighted congruences}
\label{appdec}
\subsection{Double shuffle relations}
In a recent work \cite{Jar16}, Jarossay establishes several families of $p$-adic series identities involving weighted multiple harmonic sums. We describe one of these families. Some setup is required to describe the result precisely.

Identify compositions with words on non-commuting symbols $x$, $y$ that do not end in $x$, via
\[
(s_1,\ldots,s_k)\leftrightarrow\underbrace{x\cdots x}_{s_1-1}y\;\underbrace{x\cdots x}_{s_2-1}y\cdots \underbrace{x\cdots x}_{s_k-1}y.
\]
Thus there are $|\bs|$ symbols in the word corresponding to $\bs$, and $\ell(\bs)$ of those symbols are $y$. Given two such words, we can look at words obtained by ``shuffling'' their letters. For example, there are three ways to shuffle $y$ and $xy$, yielding $yxy$, $xyy$, and $xyy$. Define $\hp(\bs_1\s\bs_2)$ to be the sum of terms $\hp(\bs_3)$ as $\bs_3$ runs through the shuffles of $\bs_1$ and $\bs_2$. For example, $(1)\leftrightarrow y$, $(2)\leftrightarrow xy$, $(1,2)\leftrightarrow yxy$, and $(2,1)\leftrightarrow xyy$, so
\begin{align}
\label{defs}
\hp\big((1)\s (2)\big)&:=\hp(1,2)+2 \hp(2,1).
\end{align}
\begin{theorem}[\cite{Jar16}, Fact 3.6]
\label{thds}
Let $\bs=(s_1,\ldots,s_k)$ and $\bt=(t_1,\ldots,t_m)$ be compositions. Then there is a convergent $p$-adic series identity
\begin{gather}
\hspace{-80mm}\hp(\bs\s\bt)=\label{gat1}\\(-1)^{|\bt|}\hspace{-4mm}\sum_{a_1,\ldots,a_m\geq 0}\prod_{i=1}^m{a_i+t_i-1\choose t_i-1}\cdot\hp\big(t_m+a_m,\ldots,t_1+a_1,s_1,\ldots,s_k\big).
\end{gather}

\end{theorem}
The proof involves an expression for the generating function of the sequence $H_n(\bs)$ as an iterated integral (the word corresponding to $\bs$ encodes the iterated integral representation). The proof also uses a substitution $n_i\leftrightarrow p-n_i$, and a $p$-adically convergent power series expansion
\[
\frac{1}{(p-n_i)^{s_i}}=(-1)^{s_i} \sum_{t\geq 0}{s_i-1+t\choose s_i-1}\frac{p^t}{n_i^{s_i+t}}.
\]

If we fix a non-negative integer $n$, we can exclude terms of weight at least $n$ from \eqref{gat1} to obtain a weighted congruence modulo $p^n$.
\begin{conjecture}
\label{conjds}
Every weighted congruence can be deduced from Theorem \ref{thds}.
\end{conjecture}

Suppose we are given a positive integer $n$, compositions $\bs_1,\ldots,\bs_k$ and coefficients $c_1,\ldots,c_k\in\Q$, and we want to determine whether the weighted congruence
\begin{equation}
\label{wctoprove}
\sum_{i=1}^k c_i\hp(\bs_i)\equiv 0\mod p^n
\end{equation}
holds for all but finitely many $p$. For each triple $\bs$, $\bt$, $\bu$ of compositions, we compute the difference between the left hand side and the right hand side of \eqref{gat1}, multiply the difference by $\hp(\bu)$ (expanding with the stuffle product), and discard those terms whose weight is $n$ or greater to obtain a weighted congruence modulo $p^n$. This is a finite computation for fixed $\bs$, $\bt$, $\bu$, and we may restrict attention to the finite collection of triples for which $|\bs|+|\bt|+|\bu|<n$. We then check whether \eqref{wctoprove} is a $\Q$-linear combination of the weighted congruences thus obtained, which is again a finite computation. If so, then we have found a proof of \eqref{wctoprove}. If not, the truth of Conjecture \ref{conjds} would imply \eqref{wctoprove} fails for infinitely many $p$.

\begin{remark}
In \cite{Ros16b} we describe a variant of this algorithm. The variant involves a formula expressing multiple harmonic sums as infinite, $p$-adically convergent linear combinations of $p$-adic multilpe zeta values. This variant is more efficient in practice, as relations among multiple zeta values have been tabulated. Our software uses the variant.
\end{remark}

\ack
The author thanks Jeffrey Lagarias for many helpful suggestions. The author thanks Shin-Ichiro Seki for pointing out an inaccuracy in the statement of Congruence \ref{propape}.

\bibliographystyle{halpha}
\bibliography{jrbiblio}
\end{document}